\documentclass[12pt,reqno]{amsart}

 \usepackage{comment}
\usepackage{tikz}
\usepackage{latexsym}
\usepackage{graphicx}
\usepackage{amssymb}
 \graphicspath{{Figures/}}
 \usepackage[show]{ed}
 \usepackage{soul}

\renewcommand{\Re}{{\operatorname{Re}\,}}
\renewcommand{\Im}{{\operatorname{Im}\,}}
\newcommand{\Op}{\operatorname{Op}}

\renewcommand{\epsilon}{\varepsilon}

\newcommand{\N}{{\mathbb N}}
\newcommand{\R}{{\mathbb R}}
\newcommand{\C}{{\mathbb C}}

\newcommand{\Z}{{\mathbb Z}}

\newcommand{\lan}{\left\langle}
\newcommand{\ran}{\right\rangle}
\newcommand{\mc}[1]{\mathcal{#1}}
\newcommand{\e}{\epsilon}

\newcommand{\re}{\mathbb{R}}
\newcommand{\Cc}{C_c^\infty}
\newcommand{\Id}{\operatorname{Id}}

\newcommand{\dbar}{\bar\partial}

\newcommand{\supp}{{\operatorname{supp\,}}}
\renewcommand{\phi}{\varphi}

\newcommand{\ep}{\varepsilon}

\newtheorem{theo}{{\sc Theorem}}

\newtheorem{cor}{{\sc Corollary}}[section]

\newtheorem{lem}[cor]{{\sc Lemma}}

\newtheorem{prop}[cor]{{\sc Proposition}}
\numberwithin{equation}{section}
\newcounter{remcounter}
\renewcommand{\theremcounter}{\arabic{remcounter}}

\newenvironment{rem}[2]{%
\refstepcounter{remcounter}%
\label{#1}%
\medskip\noindent\it{Remark~\theremcounter.}%
}%
{\medskip}%

\newtheorem{defn}[theo]{{\sc Definition}}

\usepackage{hyperref}

\title[Reverse Agmon estimates]{Reverse Agmon Estimates  in forbidden regions}

\author{John A. Toth and Xianchao Wu}
\address{Department of Mathematics and Statistics, McGill University, Montr\'eal, Canada}

\date{}

\begin{document}
  
\maketitle
\begin{abstract}  Let $(M,g)$ be a compact, Riemannian manifold and $V \in C^{\infty}(M; \R)$. Given a regular energy level $E > \min V$, we consider $L^2$-normalized eigenfunctions, $u_h,$ of the Schr\"{o}dinger operator $P(h) = - h^2 \Delta_g + V - E(h)$ with $P(h) u_h = 0$ and $E(h) = E + o(1)$ as $h \to 0^+.$ The well-known Agmon-Lithner estimates \cite{Hel} are exponential decay estimates (ie. upper bounds)  for  eigenfunctions in the forbidden region $\{ V>E \}.$ The decay rate is given in terms of the Agmon distance function $d_E$ associated with the degenerate Agmon metric $(V-E)_+ \, g$ with support in the forbidden region.

The point of this note is to prove a  reverse Agmon estimate (ie. exponential {\em lower} bound for the eigenfunctions) in terms of Agmon distance in the forbidden region under a  control assumption on eigenfunction mass in the allowed region $\{ V< E \}$ arbitrarily close to the caustic $ \{ V = E \}.$
We then give some applications to hypersurface restriction bounds for eigenfunctions in the forbidden region along with corresponding nodal intersection estimates.

\end{abstract}

\section{Introduction}

Let $(M,g)$ be a compact, $C^{\infty}$ Riemannian manifold and $V \in C^{\infty}(M;\R)$ be a real-valued potential. We assume that $E$ is a regular value of $V$ so that $dV |_{V=E} \neq 0$.
The corresponding classically allowed region is
\begin{equation} \label{allowed}
\Omega_E := \{ x \in M; V(x) \leq E \} 
\end{equation}\
\noindent with boundary $C^{\infty}$ hypersurface (ie. boundary caustic)
\begin{equation}\label{caustic}
\Lambda_E := \{ x \in M;  V(x) = E \}. \end{equation}
The {\em forbidden region} is the complement $\Omega_E^c = \{x \in M; V(x) >E \}.$

\subsubsection{Agmon-Lithner estimates} Let $P(h): C^{\infty}(M) \to C^{\infty}(M)$ be the Schr\"odinger operator
$$P(h):= - h^2 \Delta_g + V(x) - E(h)$$ and $u_h \in C^\infty(M)$ be $L^2$-normalized eigenfunctions with eigenvalue $E(h) = E + o(1)$ as $h \to 0^+$ so that
$P(h) u_h = 0.$ The {\em Agmon metric} associated with  $P(h)$ is defined by
$$ g_E (x) := (V(x)-E)_{+} \, g(x).$$
The degenerate metric $g_E$ is supported in the forbidden region $\Omega_E^c$ and we denote the corresponding Riemannian distance function by $d_E: \Omega_E^c \times \Omega_E^c \to \R^+.$  By a slight abuse of notation, we define the associated distance function to $\Lambda_E$ by

\begin{equation} \label{agmondistance}
d_E(x):= d_{E}(x, \Lambda_E) = \inf_{y \in \Lambda_E} d_{E}(x,y), \quad x \in \Omega_E^c.\end{equation}

It is well-known \cite{Hel}(3.2.2) that $d_E \in Lip(\Omega_E^c)$ and also,
$ |\nabla_x d_E|^2_g\leq (V(x)-E)_+ , a.e.$\\

Given an open subset, $U$,  of the forbidden region $ \Omega_E^c$ with $\overline{U} \subset \Omega_E^c,$ the {\em Agmon-Lithner estimate}  \cite{Hel} (Prop. 3.3.1) says that for any $\delta >0,$

\begin{equation} \label{agmon}
\| e^{(1-\delta) d_E/h} \, u_h \|_{H_h^1(U)} = O_{\delta}(1).\end{equation}\

\noindent where $ \| f \|_{H_h^1}^2 = \int_{U} ( |f|^2 +  |h \partial f|^2 ).$ A standard argument with Sobolev estimates  \cite{Hel} (Prop. 3.3.4)  then yields corresponding pointwise upper bounds: In terms of local coordinates, for any $\delta >0$ and multi-index $\alpha = (\alpha_1,...,\alpha_n) \in \N^n,$

\begin{equation} \label{agmonptwise}
| D_{x}^{\alpha} u_h (x)| \leq C_{\alpha,\delta} e^{-d_E(x)/h} e^{\delta/h}; \quad x \in \Omega_E^c. \end{equation} \

Such estimates have widespread  applications to tunnelling problems \cite{CS, Sim, HS1} and the theory of  Morse-Witten complexes \cite{Wit}.

Our objective here is to establish a partial  lower bound that is consistent with (\ref{agmon}) in a Fermi neighbourhood of the caustic $\Lambda_E$   (see subsection \ref{Fermi})  under a suitable control assumption on eigenfunction mass in the  allowed  region $\Omega_E.$ This is precisely the point of Theorem \ref{thm1}. We then give applications to lower bounds for $L^p$-restrictions of eigenfunctions to hypersurfaces in the forbidden region (so-called {\em goodness} estimates in the terminology of Toth and Zelditch \cite{TZ}).
Finally, we apply these rather explicit bounds to improve on the nodal intersection bounds of Canzani and Toth \cite{CT} for a large class of hypersurfaces in forbidden regions. We now describe our results in more detail.


Since by the monotonicity assumption, $\nabla V |_{V=E} \neq 0$, it follows that the caustic $\Lambda_{E} = \{ V = E \}$  is a $C^{\infty}$ hypersurface of $M$. 
Fix a constant $r_0 \in (0, \frac{ \text{inj} (M,g)}{2} )$ and let $U_E(r_0)$ be a Fermi neighbourhood  (see \cite{Gr} Chapter 2, section 2.1)   of the caustic $\Lambda_E$ of  width $ 2 r_0$ with respect to the ambient metric $g.$  We denote the  Fermi defining function $y_n: M \to \R$ with the property that $y_n >0$ in the forbidden part and $\Lambda_E = \{ y_n =0 \}.$ In terms of Fermi coordinates, the collar neighbourhood is $U_E(r_0) = \{ y ; |y_n| < r_0 \}.$ Consider an annular region in $U_{E}(r_0) \cap \{ V>E \}$ given by $A(\delta_1,\delta_2):= \{ y \in U_E(r_0);  E+ \delta_1 < V(y) < E + \delta_2 \}$ with $0 < \delta_1 < \delta_2.$ Our first result in Theorem \ref{thm1} is a partial  reverse  Agmon estimate consistent with  (\ref{agmon}). First, we introduce a {\em control } assumption on the eigenfunctions $u_h$ in the allowed region. 
\begin{defn}\label{control} We say that the eigenfunctions $u_h$ satisfy the {\em control assumption} if for every $\epsilon>0$ there exists  constants $C_N(\epsilon)>0$ and $h_0(\epsilon)>0$ so that for $h \in (0,h_0(\epsilon)],$
\begin{equation} \label{controleqn}
 \int_{ \{ E- \frac{\epsilon}{2}  \leq V(x) \leq E  \} } |u_h|^2 \, dv_g \geq C_N(\epsilon) h^N \end{equation}
for some $ N>0.$ When (\ref{controleqn}) is satisfied for a {\em fixed} $\epsilon = \epsilon_0>0,$ we say that the eigenfunction sequence satisfies the {\em $\epsilon_0$ control assumption}.
\end{defn}

Roughly speaking, the control assumption in Definition \ref{control} says that in arbitrarily small (but independent of $h$) annular neighbourhoods of the caustic in the {\em allowed} region, eigenfunctions have at least polynomial mass in $h.$


Our second assumption involves a monotonicity condition on the potential V itself in the tubular neighbourhood $U_E(r_0)$. Specifically, we make the following

\begin{defn} \label{monotonicity}
Given $r_0  \in (0, \frac{inj(M,g)}{2})$, we say that $V$ satisfies the   {\em monotonicity assumption}   in $U_E(r_0)$ provided there exists $C(r_0)>0$ such that
$$ \,\, \partial_{y_n} V(y) \geq C(r_0),\quad \,\, y \in U_E(r_0).$$
\end{defn}

We note that the control assumption is automatically satisfied in the 1D case as a consequence of the  WKB  asymptotics for the eigenfunctions. In section \ref{counterexample}, we give examples of eigenfunction sequences satisfying this condition in arbitrary dimension.

As for the monotonicity condition in Definition \ref{monotonicity},  it is not hard to show (see section \ref{collar}  (\ref{defining})) that for $r_0>0$ sufficiently small, this condition is necessarily satisfied in $U_E(r_0)$  provided  $\nabla V |_{ \{V=E\} } \neq 0$ for the energy value $E.$  We note that this condition can readily be written in a more geometrically intrinsic way in terms of the local normal foliation of the tubular neighbourhood of $\Lambda_E$ (see subsection \ref{collar}).

Then,  under the control and monotonicity assumptions  in Definitions \ref{control}  and \ref{monotonicity}, by using Carleman estimates to pass across the caustic hypersurface, in Theorem \ref{thm1} we prove that for any $\epsilon >0$ and $h \in (0,h_0(\epsilon)],$

\begin{equation} \label{upshot1}
\| e^{\tau_0 d_E/h} \, u_h \|_{H_h^1 (A(\delta_1,\delta_2))} \geq C(\epsilon,\delta_1,\delta_2) e^{-\beta(\epsilon)/h}, \end{equation}\

\noindent where $ \beta(\epsilon) = o(1)$
as $\epsilon \to 0^+$ and  

$$ \tau_0 := \Big( \frac{ \max_{y \in U_E(r_0)} |\partial_{y_n} V|}{ \min_{y \in U_E(r_0)\cap\Omega_E^c} |\partial_{y_n} V|} \Big)^{1/2}.\ $$\

\begin{rem}{rem:A0}{} \label{comments} 

 (i) The $\epsilon$-dependence on the RHS of (\ref{upshot1}) is a consequence of the control assumption on the eigenfunctions $u_h$ in Definition \ref{control}. \  

(ii) It is easy to see that the control assumption in Defintion \ref{control} is necessary in Theorem \ref{thm1} since simple  counterexamples can be constructed otherwise arising from the natural occurence of additional effective potentials upon separation of variables (see section \ref{counterexample}). \

  (iii) The control assumption in Definition \ref{control} can be weakened slightly in that the polynomial mass assumption in $h$ can be replaced by the following subexponential one: for all $\epsilon >0,$
 $$ \liminf_{h \to 0^+}  \, h \, \log \Big( \int_{ \{ E- \frac{\epsilon}{2}  \leq V(x) \leq E  \} } |u_h|^2 \, dv_g \Big) = 0.$$
It follows easily (see (\ref{UPSHOT})) that Theorem \ref{thm1} and consequently, Theorems \ref{thm2} and \ref{thm3} all hold under this slightly weaker assumption.

(iv) In the case where the eigenfunction sequence only satisfies the  $\epsilon_0$-control assumption, the lower bound in (\ref{upshot1}) is also satisfied, where the constant $\beta(\epsilon_0)>0$ appearing on the RHS of the inequality can be explicitly estimated in terms of the potential, $V$ (see Remark \ref{general}). The same is true for the subsequent results in Theorems \ref{thm2} and \ref{thm3}. \end{rem}

Clearly, the geometric constant $\tau_0 \geq 1.$  At present, we are unable to prove that (\ref{upshot1}) holds in the general setting above with optimal constant $\tau_0 =1,$  but we hope to return to this point elsewhere.\

We note that in one-dimension, both upper and lower bounds for eigenfunctions are well-known; indeed, classical WKB methods \cite{E} (Chapter 4, sections 4.3-4.7) provide explicit asymptotic formulas for  the eigenfunctions. In higher dimensions, the situation is much more complicated and, to our knowledge, the problem of establishing lower bounds has only been addressed in special cases associated with low-lying eigenvalues (i.e. semi-excited states) \cite{Hel} (section 4.4) and \cite{Sim}.
For more general eigenfunctions (eg. excited states), while the upper bounds given by the Agmon estimates are well-known, to our knowledge explicit lower bounds have not been addressed in the literature. That is the main point of Theorem \ref{thm1} above.  \

In section \ref{restriction} we use the Carleman bounds in (\ref{upshot1}) with shrinking annuli together with a Green's formula argument to get lower bounds for $L^p$ eigenfunction restrictions to hypersurfaces smoothly isotopic in $U_E(r_0) \cap \{ V>E \}$ to level sets $H = \{y_n = const.\}$ (see Definition \ref{admissible}).
In case of the level sets $H= \{y_n = const\}$, Theorem \ref{thm2} says that, under the same assumptions as in (\ref{upshot1}), for any $\epsilon >0$ and $h \in (0,h_0(\epsilon)]$ and with
$$d_E^H := \max_{y \in H} d_E(y), \,\,\,\ d_E(H):= \min_{y \in H} d_E(y),$$

\begin{equation} \label{upshot2}
\| u_h \|_{L^p(H)} \geq C(p,\epsilon) e^{- ( \, 2 \tau_0 d_E^H - d_E(H) \, ) / h} \, e^{-\beta(\epsilon)/h}, \quad p \geq 1,\end{equation}\

\noindent where $\beta(\epsilon) = o(1) \,\, \text{as} \,\, \epsilon \to 0^+.\ $ \\

The bounds in (\ref{upshot2}) are {\em goodness} estimates in the terminology of Toth and Zelditch \cite{TZ}; the key novelty here is the rather explicit geometric rate  $2 \tau_0 d_E^H - d_E(H)$ appearing in (\ref{upshot2}).


Finally, in section \ref{nodal}, we give an application of (\ref{upshot2}) to nodal intersection bounds in forbidden regions. In \cite{CT}, Canzani and Toth prove that for any $C^{\omega}$ separating hypersurface $H$ in the forbidden region, with $Z_{u_h} = \{ x \in M; \,  u_h(x) = 0 \},$
$$ \# \{ Z_{u_h} \cap H \} \leq C_H' h^{-1}.$$

While this rate in $h$ is easily seen to be sharp in general,   the constant $C_H'>0$ is not explicitly controlled in  \cite{CT}.  The lower bound in (\ref{upshot2}) allows us to give a more concrete upper bound  for $C_H'$  in the cases where $H$ is   a separating hypersurface  that is smoothly isotopic to a level set of the defining function $y_n$ in the forbidden region. This is essentially the content of Theorem~ \ref{thm3}. 

Finally, we note that while all results are stated here for compact manifolds, the results in Theorems \ref{thm1}-\ref{thm3} extend to the case of Schr\"odinger operators on $\R^n$ and the proofs are the same.

\noindent{\em Acknowledgements:} We would like to thank Jeff Galkowski, Andreas Knauf  for many helpful discussions. We also thank Stephane Nonnenmacher and the referees for detailed comments regarding earlier versions of the paper.

\section{Carleman estimates in a Fermi neighbourhood of the caustic} \label{Fermi}









\subsection{Collar neighbourhood of caustic and Fermi coordinates} \label{collar}

For background on the construction of Fermi coordinates in a tubular neighbourhood of any embedded compact submanifold $P \subset M,$ we refer the reader to \cite{Gr} Chapter 2 (see also \cite{HHHZ} section 2). Here, given a compact Riemannian manifold $(M,g)$, we are interested in the special case where $P$ is a $C^{\infty}$ compact hypersurface; specifically, $P= \Lambda_E = \{ V=E \}.$ 
Let $\nu \in N \Lambda_E$ be the unit, oriented vector field normal to $\Lambda_E.$ Then, by the tubular neighbourhood theorem, there exists an open neighbourhood, $U_E(r_0),$ of $\Lambda_E$ such that
$$ \Lambda_E \times (-r_0, r_0) \ni (y',t) \mapsto \exp_{y'}(t \nu) \in U_E(r_0)$$
is a $C^{\infty}$ diffeomorphism provided $r_0 < inj(M,g)/2,$ where $\exp$ denotes the geodesic exponential map in the metric $g.$ Then, the Fermi defining function $y_n: U_{E}(r_0) \to (-r_0, r_0)$ is given by 
$$y_n (\exp_{y'}(t \nu)) = t; \quad t \in (-r_0,r_0).$$
Writing $\gamma_{\nu}(t,y'): = \exp_{y'}(t \nu),$ it then follows that the monotonicity condition in Definition \ref{monotonicity} can be rewritten in the somewhat more geometric form
\begin{equation} \label{monotone2}
\langle \nabla V, \dot{\gamma_{\nu}}(t,y') \rangle_g \geq C_0 >0, \quad (t,y') \in (-r_0,r_0) \times \Lambda_E. \end{equation}

 
 We will assume from now on that the monotonicity assumption in Definition \ref{monotonicity} (or equivalently, (\ref{monotone2})) is satisfied (see Remark \ref{wlog}).    In the following, we identify points $y' \in \Lambda_E$ with their local coordinate representations in $\Lambda_E.$ Then, in terms of   Fermi coordinates $(y', y_n)$ in the ambient metric $g,$ we have 

$$ g = dy_n^ 2 + \sum_{i,j =1}^{n-1} h_{ij}(y',y_n) dy'_{i} dy'_{j}, \quad y \in U_E(r_0)$$\
  
\noindent  where $ \sum_{i,j=1}^{n-1} h_{ij}(y',0) dy'_i dy'_j$ is the induced metric on $\Lambda_E.$  In the following, we abuse notation somewhat and simply write $h(y',y_n) |dy'|^2 : = \sum_{i,j=1}^{n-1} h_{ij}(y',y_n) dy'_{i} dy'_{j}.$  
  
In the following,  $ y_n  \in C^{\infty}(M;\R),$ is the Fermi defining function for $\Lambda_{E}$ defined above  with
$$ \Lambda_E = \{ y_n = 0  \}, \quad dy_n |_{\Lambda_E} \neq 0.$$
\noindent 
 Since under the monotonicity assumption, $\partial_{y_n} V(y) >0,$ it is immediate that then $\nabla V(y) \neq 0$ for $y \in U_E(r_0).$  Thus, it follows that $V(y)-E$ is  also a legitimate defining function in the tubular neighbourhood $U_E(r_0).$

We choose our sign convention so that 
\begin{equation} \label{sign}
\{ V > E \} \cap U_E(r_0) = \{y; 0 <  y_n  < r_0 \} \,\,\, \text{and} \,\,\, \{ V < E \} \cap U_E(r_0) = \{y; - r_0 <  y_n  < 0 \}. \end{equation}\

It will also be useful in the following to introduce the following annular domains in the forbidden region defined by

\begin{equation} \label{annulus}
A(\delta,\delta') :=  \{x \in M;  \delta < y_n(x) <  \delta'  \}, \quad  0 < \delta < \delta'.\end{equation} \

In terms of the Fermi coordinates $y=(y',y_n),$ the corresponding Agmon metric $g_E |_{U_E(r_0)}$   has the form

\begin{eqnarray} \label{agmonlocal0}
g_E = ( V(y) - E ) \,( dy_n^2 + h(y',y_n) |dy'|^2 ), \quad y \in U_E(r_0), \, y_n >0.
 \end{eqnarray}\


 

 
 

 

\begin{rem}{rem:A}{} \label{wlog}  By first-order Taylor expansion, 
\begin{eqnarray} \label{defining}
 V(y) - E =y_n \, F(y',y_n), \end{eqnarray}
 where
$$ F(y',y_n) =\int_{0}^{1} (\partial_{y_n}V) (y', t y_n) \, dt.$$

Differentiation of (\ref{defining}) in the $y$-variables, gives

$$ \partial_{y_n} V(y) = F(y',y_n) + y_n \partial_{y_n}F(y), \,\,\, \partial_{y'} V(y) = y_n \partial_{y'}F(y',y_n), \,\,\, F(y',0) = \partial_{y_n}V(y',0).$$ \

Thus, $ \partial_{y'} V(y',0) = 0$ and since $\nabla V |_{V=E} \neq 0$ it follows that $\partial_{y_n} V(y',0) \neq 0.$  Then, by continuity of $\nabla V$ there exists $r_0>0$ so that $| \partial_{y_n} V(y) | \geq C >0$ for $y \in U_E(r_0)$ and under the  sign conventions in (\ref{sign}), 
$$ \partial_{y_n} V(y) \geq C >0, \quad y \in U_E(r_0).$$  

Consequently, when $E \in \R$ is a regular value of $V$, one can always find a constant $r_0 >0$ so that the monotonicity assumption in Definition \ref{monotonicity} is satisfied.
\end{rem}

Under the monotonicity condition, we then have $F(y) \geq C >0$ for $y \in U_E(r_0)$  and  from (\ref{defining}), 

\begin{equation} \label{useful} 
\min_{y \in U_E(r_0) \cap \Omega_E^c } \partial_{y_n} V(y) \leq F(y) \leq \max_{y \in U_E(r_0) \cap \Omega_E^c} \partial_{y_n} V(y); \quad y \in U_E(r_0) \cap \Omega_E^c. \end{equation}

Finally, we note that in terms of the decomposition (\ref{defining}), the  Agmon metric can  be written in the form

\begin{eqnarray} \label{agmonlocal}
g_E =  y_n \, F(y) \,( dy_n^2 + h(y',y_n) |dy'|^2 ), \quad y \in U_E(r_0), \,\,\, y_n >0, 
 \end{eqnarray}\

\subsubsection{Locally minimal geodesics and Agmon distance}

In the collar neighbourhood $U_E(r_0),$ given a point $(y',y_n) \in U_{E}(r_0)\cap \Omega_E^c,$ there is a unique minimal  (i.e. distance-minimizing)  geodesic $\gamma: [0,1] \times \Lambda_E \to U_E(r_0)$ for the ambient metric $g$ joining $(y',y_n)$ to the caustic hypersurface $\Lambda_E.$ Setting $\gamma_t(y') = \gamma(t,y')$ where $\gamma_0(y') = (y',0) \in \Lambda_E$ and $\gamma_1(y') = (y',y_n),$ the minimal geodesic is just
$$ \gamma_t(y') = (y', ty_n); \quad 0 \leq t \leq 1.  $$ 
In terms of the discussion in subsection \ref{collar} above, these are just the geodesic segments normal to $\Lambda_E.$ 
It is easy to see that these normal geodesic segments to $\Lambda_E$ are unfortunately not, in general, minimal geodesics for the conformally rescaled Agmon metric $g_E$; indeed the latter can be quite complicated. Nevertheless, we will need the following elementary estimate for Agmon distance in terms of the natural Fermi defining function $y_n: M \to \R$ above.

\begin{lem} \label{estimate}  Let $E$ be a regular value of $V \in C^{\infty}(M)$ and assume that the monotonicity condition is satisfied in $U_E(r_0).$ Then, 
$$ d_E(y) \geq  \frac{2}{3} \,  ( \, \min_{y \in U_E(r_0)\cap\Omega_E^c} \partial_{y_n} V(y) \,)^{1/2} \,\, y_n^{3/2}; \quad y \in   U_E(r_0)\cap\Omega_E^c.$$
\end{lem}

\begin{proof}
Let $\gamma: [0,1] \in \Omega_E^c$ be a piecewise-$C^1$ minimal geodesic for the Agmon metric $g_E$ joining $y = (y',y_n) \in U_E(r_0)\cap \Omega_E^c$ to $\Lambda_E;$ explicitly, $\gamma(0) = (y',y_n)$ and $\gamma(1) = (f(y',y_n), 0)$ where $f(y) \in \R^{n-1}.$ Then, writing $\gamma = (\gamma', \gamma_n),$ with $\gamma' = (\gamma_1,...,\gamma_{n-1}),$
$$ d_E(y) = \int_{0}^1 |d_t\gamma(t)|_{g_E} \,dt,$$
and since
 $$|d_t \gamma(t)|_{g_E} = \Big( \, F(\gamma(t)) \, \gamma_n(t)   \, |d_t \gamma_n(t)|^2 + F(\gamma(t)) \, \gamma_n(t) \,  \langle h(y(t)) \,  d_t \gamma'(t), \, d_t\gamma'(t) \rangle \, \Big)^{1/2},$$
  with $F, \gamma_n > 0,$  and $0 \leq h \in GL(n-1,\R),$ it follows that
  
$$d_E(y) \geq   \min_{y \in   U_E(r_0)\cap\Omega_E^c} F(y)^{1/2} \cdot \int_0^1 \gamma_n(t)^{1/2} \, |d_t \gamma_n(t)| \, dt, \quad  y \in   U_E(r_0)\cap\Omega_E^c.$$

Finally, by making the change of variables $t \mapsto s = \gamma_n(t)$ in the last integral, one gets
$$ d_E(y) \geq \min_{y \in   U_E(r_0)\cap\Omega_E^c} F(y)^{1/2} \cdot \int_{0}^{y_n} s^{1/2} \, ds, \quad  y \in   U_E(r_0)\cap\Omega_E^c.$$
and the lemma follows from this last estimate combined with (\ref{useful}) since $\min_{y \in   U_E(r_0)\cap\Omega_E^c} F(y)^{1/2} \geq \min_{y \in   U_E(r_0)\cap\Omega_E^c}  \,( \, \partial_{y_n}V (y) \,)^{1/2} .$  
\end{proof}




\subsection{Local control and Carleman bounds near the caustic $\Lambda_E$}

\subsubsection{Model computation}

Consider the model Airy operator $P_0(h) := (hD_y)^2 + y$ where $y \in \R$ where
$V(y) = y$ and  $E =0$ with the corresponding Airy-type weight function in the forbidden region given by
$$\phi_0(y) = \frac{2}{3} \, y^{ 3/2}, \quad y >0.$$

Then, the symbol of the conjugated operator $e^{\phi_0/h} P_0(h)  e^{-\phi_0/h}$ is 
$$p_{\phi_0}(y,\xi) = \xi^2 - |\phi'_0(y)|^2 + y + 2  i y^{1/2} \xi, \,\,\, y>0 $$ and
$$ \text{Char}(p_{\phi_0}) = \{ (y,\xi) \in \R^2; \xi = 0, y>0 \}.$$\

The latter follows since $(y,\xi) \in  \text{Char}(\phi_{\phi_0})$ iff
  $0 = |\xi|^2 - |\phi_0'(y)|^2 + y  + 2i y^{1/2} \xi$ which in turn holds iff $\xi = 0$ provided $y>0,$
  since $|\phi_0'(y)|^2 - y = 0.$ 
  \medskip

We note that the weight function $\phi_0$
is {\em borderline} for the  H\"{o}rmander subelliptic condition in the sense that
for $(y,0) \in \text{Char} \, (p_{\phi_0})$, we have

\begin{equation} \label{borderline}
\{ \Re p_{\phi_0}, \Im p_{\phi_0} \} = 4 \phi_0''(y) |\phi_0'(y)|^2 - 2 \phi_0'(y) \equiv 0,  \quad y >0. \end{equation} \

Of course, in this case, $\phi_0(y) = \frac{2}{3} y^{3/2} = \int_{0}^{y} \tau^{1/2} d\tau $ is precisely the Agmon distance function $d_E(y),$ where by convention we have set $E=0.$   In view of (\ref{borderline}), it is reasonable to expect that one can construct an appropriate perturbation, $\phi_{\epsilon}$, of the Airy model weight function $\phi_0$, that is a Carleman weight with $ \{ \Re p_{\phi_{\epsilon}}, \Im p_{\phi_{\epsilon}} \} >0$ in a neighbourhood of the caustic $\Lambda_{E}.$  Indeed, as we show in subsection \ref{Carleman} below, one can readily construct such a $\phi_{\epsilon}$ as a legitimate Carleman weight in a Fermi neighbourhood of $\Lambda_{E}.$  Combined with the control assumption on the eigenfunctions $u_h$, we will prove Theorem \ref{thm1} by applying Carleman estimates in a Fermi neighbourhood of the caustic with weignt $\phi_{\epsilon}.$

\subsubsection{Construction of the weight function} \label{Carleman}
Let $P(h) = -h^2 \Delta_g + V - E: C^{\infty}(M) \to C^{\infty}(M)$ and consider the conjugated operator 
$P_{\phi}(h) = e^{\phi/h} P(h) e^{-\phi/h}: C^{\infty}(M) \to C^{\infty}(M)$  with principal symbol 
$p_{\phi}(x,\xi) = |\xi|^2_g - |\nabla_x \phi|^2_g + V(x) - E+2i\left<\xi, \nabla_x \phi\right>_g.$
The model case above suggests that to create subellipticity for $P_{\phi}(h)$ in a Fermi neighbourhood of the caustic, it  should suffice to slightly modify the model weight function $\phi_0$ in the normal Fermi coordinate $(y',y_n).$ With this in mind, for $\epsilon >0$ arbitrarily small (for concreteness, assume $10 \epsilon < r_0$) and constant $\tau >0$ to be determined later on, we now set in Fermi coordinates $(y',y_n): U_E \to \R^{n-1} \times (-r_0, r_0),$

\begin{equation} \label{localweight}
\phi_{\epsilon}(y_n) :=   \Big( \frac{2}{3} + \epsilon \Big)  \, \tau \,  ( y_n + 10 \epsilon )^{3/2}, \quad  y_n \in (-4\epsilon, r_0). \end{equation}\

\begin{rem}{rem:B}{} We recall here that $r_0 < \text{inj}(M,g)$ is fixed (but not necessarily small), whereas $\epsilon >0$ will be chosen arbitrary small (but independent of $h$) consistent with the control assumption on the eigenfunctions.
\end{rem}

We abuse notation somewhat in the following and simply write $\phi= \phi_{\epsilon},$ the dependence on $\epsilon$ being understood.
Then, $\phi \in C^{\infty}([-4\epsilon,r_0])$ and plainly $\phi: [-4\epsilon, r_0] \to \R^+$  is strictly-convex and monotone increasing with

$$ \min \, ( \,\phi'(y_n), \phi''(y_n) \,) \geq C(\epsilon) >0, \quad  y_n \in (-4\epsilon, r_0).$$\

Let $\pi: T^*M \to M$ be the natural projection map $\pi(y,\xi)= y.$  Then, the relevant characteristic variety is
$$  \text{Char}(p_{\phi}) \cap \pi^{-1} ( \{ (y',y_n);  y_n \in (-4\epsilon, r_0) \} ) = \{ (y,\xi);  | \xi |_y^2 - |\partial_{y_n} \phi|^2 + F(y) y_n = 0,  \, \xi_n  = 0, \,\, y_n \in (-4\epsilon, r_0) \}.$$\

Since $F(y)>0,$ it follows that this set is non-trivial; indeed for any $-4\epsilon < y_n <0$ (ie. a point in the allowed region),
$$   \text{Char}(p_{\phi}) \cap \pi^{-1} ( \{(y',y_n), y_n \in (-4\epsilon, 0) \}) \cong S^*M_{A(-4\epsilon, 0)}  \cap \{ \xi_n = 0 \},$$

\noindent where $S^*M_{A(-4\epsilon, 0)}:= \{ (y,\xi) \in T^*M; |\xi|_y = 1, \,\, y \in A(-4\epsilon, 0) \}.$

Since Char$(p_{\phi})$ is non-trivial, global ellipticity over the interval $(-4\epsilon, r_0)$ evidently fails. However, we claim that  {\em subellipticity} is now satisfied in such an interval provided $\tau >0$ is chosen large enough but depending only on the potential $V$. Indeed, since the normal Fermi coordinate is $y_n $ and $\phi$ is a function of {\em only} $y_n$ with $g_{n,n} = 1,$ a direct computation gives,

\begin{align} \label{guts}
\{ \Re p_{\phi}, \Im p_{\phi} \} &= \{ \xi_n^2 + |\xi'|^2_y - (\partial_{y_n} \phi)^2 + V - E,   \, 2 \partial_{y_n} \phi \cdot \xi_n \} \nonumber \\
&= 4 \partial_{y_n}^2 \phi  \,\Big(   |\partial_{y_n} \phi |^2 + \xi_n^2  \Big) - 2 \partial_{y_n} \phi \cdot \partial_{y_n} V \nonumber \\
&\geq 2  \partial_{y_n} \phi  \, ( \, 2  \partial_{y_n}^2\phi \cdot \partial_{y_n} \phi - \, \partial_{y_n} V \, ) \nonumber \\
&\geq 2 \tau  C(\epsilon) ( \, 2  \partial_{y_n}^2\phi \cdot \partial_{y_n} \phi - \, \partial_{y_n} V \, ), \quad y_n \in (-4\epsilon,  r_0). 
\end{align}\

From (\ref{localweight}), for any $\epsilon >0$ and  for all $ y_n \in (-4\epsilon,  r_0),$

$$ 2 \partial_{y_n}^2 \phi \cdot \partial_{y_n}\phi \equiv \frac{9}{4} \tau^2  \Big(\frac{2}{3} + \epsilon \Big)^2 > \tau^2. $$\


Choosing 
\begin{equation} \label{scaling}
\tau = \| \partial_{y_n} V \|_{L^\infty(U_E(r_0) )}^{1/2},
\end{equation} \

it follows from (\ref{guts}) that  for all $(y, \xi)$ with $y_n \in (-4\epsilon, r_0),$

$$ \{ \Re p_{\phi}, \Im p_{\phi} \}(y,\xi) \geq C(\tau,\epsilon) >0.$$\

Consequently, $\phi = \phi_{\epsilon}$ is a Carleman weight for $P(h)$ {\em globally} in the Fermi neighbourhood of the caustic where $- 4\epsilon < y_n <  r_0.$\\

 Now, let $\chi \in C^{\infty}_0(\R;[0,1])$ be a cutoff satisfying 

$$\chi(y_n) = 1; \quad -\frac{3}{2}\epsilon < y_n < \tilde{\delta_1}$$
with 
$$ \hspace{8mm} \chi(y_n) = 0; \quad  y_n \in \R \setminus  (-2\epsilon, \tilde{\delta_2}).$$\\
where $0 < \tilde{\delta_1} < \tilde{\delta_2} <  r_0.$ \\

\begin{center}
\begin{tikzpicture}
\draw[->] (-3,0) -- (-2,0) node[below]{$-2\epsilon$} -- (-1.7,0) node[below right]{-$\frac{3}2 \epsilon$} -- (-.8,0) node[below]{0} -- (4.0,0) node[below]{$\tilde{\delta_1}$} -- (6,0) node[below]{$\tilde{\delta_2}$} -- (7.5,0) node[below]{$\R$};
\draw[fill] (-.8,0) circle [radius=1pt];
\draw[fill] (-2,0) circle [radius=1pt];
\draw[fill] (-1.5,0) circle [radius=1pt];
\draw[fill] (4.0,0) circle [radius=1pt];
\draw[fill] (6,0) circle [radius=1pt];
\draw[very thick] (-2,0) to [out=0,in=180] (-1.5,1) to [out=0, in=180] (4.5, 1) to [out=0, in=180] (6, 0);
\node[above] at (10,0.5) {Figure 1. \,\,\,cutoff function $\chi$};
\end{tikzpicture}
\end{center}
\medskip


  We note that one can write 
$$ \chi = \chi_{-} \, \chi_+$$ where,  $\chi_{+} \in C^{\infty}(\R; [0,1])$ satisfies
$$ \chi_{+}(y_n) = 1, \quad  y_n <  \tilde{\delta_1}, $$
$$  \chi_{+}(y_n) = 0, \,\,\, y_n > \tilde{\delta_2},$$ 

\noindent and similarily, $\chi_{-} \in C^{\infty}(\R; [0,1])$ satisfies
$$ \chi_{-}(y_n) = 1;  \quad y_n >  -\frac{3}{2} \epsilon,$$
$$ \chi_{-}(y_n) = 0; \quad y_n  < -2 \epsilon.$$

Then, it follows from Leibniz rule that
$$ \text{supp} \, \partial \chi \subset \text{supp} \,\partial \chi_{+} \cup \text{supp} \,\partial \chi_{-},$$

\noindent where supp $\partial \chi_{-} \subset [-2 \epsilon, - \frac{3}{2} \epsilon]$ and supp $\partial \chi_{+} \subset [ \tilde{\delta_1}, \tilde{\delta_2}].$  \  \\


Set $P_{\phi}(h) := e^{\phi/h} P(h) e^{-\phi/h}: C^{\infty}_{0}(U) \to C^{\infty}_{0}(U)$ and with $\chi = \chi(y_n)$ above,
$$ v_h := e^{\phi/h} \chi  u_h$$ where $P(h) :=-h^2 \Delta_g + V(x) - E(h)$ and
$$P(h) u_h = 0.$$
Moreover, we assume throughout that the eigenfunctions $u_h$ are $L^2$-normalized with \newline $\| u_h \|_{L^2(M,g)} = 1.$

In view of the subellipticity estimate in (\ref{guts}) and the support properties of the cutoff $\chi \in C^{\infty}_0$ it follows by 
the standard Carleman estimate \cite[Theorem 7.7]{Zwo}  that

\begin{equation} \label{c1} 
\| P_{\phi}(h) v_h \|^2_{L^2} \geq C_1(\epsilon) h  \, \| v_h \|_{H_h^1}^2. \end{equation}\

Since $P(h) u_h =0$  and $P_{\phi}(h)$ is local with  $\text{supp}\,  \partial \chi_+ \cap \text{supp}\, \partial  \chi_-= \emptyset, $  it follows from (\ref{c1}) that

\begin{eqnarray} \label{c2}
 \| e^{\phi/h} [P(h), \chi] \, u_h \|_{L^2}^2  
  \geq C_1(\epsilon) h \, \| e^{\phi/h} \chi u_h \|_{H_h^1 }^2   \end{eqnarray}\
  
  Let  $\supp\widetilde{\partial \chi_{\pm}}$  be arbitrarily small neighbourhoods of $\text{supp} \, \partial \chi_{\pm}$ respectively.  Specifically, we can assume that $ \text{supp} \, \widetilde{ \partial \chi_\pm} \, \supset \, \text{supp} \, \partial \chi_{\pm},$ and choose $(\delta_1,\delta_2)$ with $0 < \delta_1 < \tilde{\delta_1} < \tilde{\delta_2} < \delta_2 $ so that $\text{supp} \, \widetilde{ \partial \chi_+} \subset [\delta_1,\delta_2]$  and in addition
$$ \text{meas} \, (  \text{supp} \, \widetilde{ \partial \chi_\pm} \setminus \text{supp} \, \partial \chi_{\pm} ) \leq \frac{\epsilon}{10}.$$

Since the support of the coefficients of the $h$-differential operator $[P(h),\chi]$ is contained in 
 $\text{supp} \, \widetilde{ \partial \chi_{+} } \cup  \text{supp} \, \widetilde{ \partial \chi_{-}},$ it follows from (\ref{c2}) that

 \begin{eqnarray} \label{c3}
 \| e^{\phi/h} [P(h), \chi]  \, u_h \|_{L^2 ( \text{supp} \,\widetilde{\partial \chi_+}) }^2 +  \,  \| e^{\phi/h}  [P(h), \chi] u_h \|_{L^2 ( \text{supp} \, \widetilde{ \partial \chi_-} )}^2 & \nonumber \\ \nonumber \\
  \geq C_1(\epsilon) h \| e^{\phi/h} \chi u_h \|_{H_h^1}^2. &\end{eqnarray}\

  Since $\chi(y_n) = 1$ when $y_n \in (-\epsilon/2, 0),$ it follows from (\ref{c3}) that
  
   \begin{eqnarray} \label{c3.5}
 \| e^{\phi/h} [P(h), \chi]  \, u_h \|_{L^2 ( \text{supp} \, \widetilde{\partial \chi_+}) }^2 +  \,  \| e^{\phi/h}  [P(h), \chi] u_h \|_{L^2 ( \text{supp} \, \widetilde{ \partial \chi_-} )}^2 & \nonumber \\ \nonumber \\
  \geq C_1(\epsilon) h \| e^{\phi/h} u_h \|_{H_h^1 (\{ y; -\epsilon/2 \leq y_n \leq 0 \} )}^2. &\end{eqnarray}\
  
  Since $e^{\phi/h} [P(h),\chi]  e^{-\phi/h} = h Q(h)$ where $Q(h) = a(x) h \partial_x + b(x), \,\, a, b \in C^{\infty}(M)$, is an $h$-differential operator of order one,

\begin{equation} \label{c3.6}
 \| e^{\phi/h} [P(h), \chi]  \, u_h \|_{L^2 ( \text{supp} \, \widetilde{\partial \chi_{\pm} }) }^2 \leq C h^2 \| e^{\phi/h} u_h \|^2_{ H_h^1 ( \text{supp} \, \widetilde{ \widetilde{\partial \chi_{\pm}} } )}. \end{equation}

\noindent where  supp  $ \widetilde{ \widetilde{\partial \chi_{\pm}} } \Supset \text{supp}\,  \widetilde{\partial \chi_{\pm}}.$ After relabelling $\widetilde{ \widetilde{\partial \chi_{\pm}} }  $ simply as $ \widetilde{\partial \chi_{\pm}}$, one gets from (\ref{c3.5}) and (\ref{c3.6}) that with an appropriate constant $C_2(\epsilon)>0,$

\begin{eqnarray} \label{c3.7}
h^2\| e^{\phi/h} u_h \|^2_{ H_h^1 ( \text{supp} \, \widetilde{\partial \chi_{+}} )} + h^2 \| e^{\phi/h} u_h \|^2_{ H_h^1 ( \text{supp} \, \widetilde{\partial \chi_{-}}  )} & \nonumber \\ \nonumber \\
\geq C_2(\epsilon) h \| e^{\phi/h} u_h \|_{H_h^1 (\{ y; -\epsilon/2 \leq y_n \leq 0 \} )}^2, &\end{eqnarray}  \ 

\noindent or equivalently,
 
 \begin{equation} \label{c4}
 h^2 \| e^{\phi/h}  u_h \|_{H^1_h( \text{supp} \, \widetilde{ \partial \chi_+})}^2   
 \geq  C_2(\epsilon) h \| e^{\phi/h}  u_h \|_{H_h^1(\{ y; -\epsilon/2 \leq y_n \leq 0 \} )}^2 -  h^2 \| e^{\phi/h} u_h \|_{H_h^1( \text{supp} \, \widetilde{ \partial \chi_{-} })}^2, \end{equation}  \\



Since $\widetilde{\partial \chi_{-}}$ is supported in the classically allowed region where $y_n<0$, we will now use the control assumption  in Definition \ref{control} to get an effective lower bound for the RHS in (\ref{c4}). \

Computing in Fermi coordinates and using the fact that one can take  supp $\, \widetilde{ \partial \chi_-} \subset \{ y; -3\epsilon < y_n < - \epsilon \},$ the RHS of (\ref{c4}) is 

\begin{eqnarray} \label{c5}
&\geq  \,C_2(\epsilon) h  \int_{ \{ y; y_n \in (-\frac{\epsilon}{2}, 0 )  \} }  e^{2\phi(y_n)/h} ( |u_h(y)|^2  + |h \partial_y u_h(y)|^2) \, dy' dy_n \nonumber \\ \nonumber \\
& - \,  h^2 \int_{ \{ y; y_n \in (- 3\epsilon, - \epsilon) \} }  e^{2\phi(y_n)/h} ( |u_h(y)|^2  + |h \partial_y u_h(y)|^2) \, dy' dy_n.
\end{eqnarray}\\


Next we use strict monotonicity of the weight function $\phi \in C^{\infty}([-3\epsilon, r_0])$ in (\ref{localweight}).
 We set $m(\epsilon):= \min_{y_n \in (-\frac{\epsilon}{2}, 0)}  \phi(y_n) >0$ and $M(\epsilon):= \max_{y_n \in  (- 3\epsilon, - \epsilon) } \phi(y_n)>0$  (we note that   both $m(\epsilon)$ and $M(\epsilon)$ are of order $\epsilon^{3/2}$). Then, since $\phi$ is strictly increasing, 
$$ m(\epsilon) - M(\epsilon) = C_3(\epsilon)>0.$$
So, it follows that (\ref{c5}) is bounded below by

\begin{eqnarray}  \label{c6}
 C_2(\epsilon) e^{2m(\epsilon)/h} \Big( \, h  \| u_h \|^2_{H_h^1( \{y; y_n \in (-\frac{\epsilon}{2}, 0 )  \} )} - C_2(\epsilon)^{-1}  h^2 e^{2[M(\epsilon)-m(\epsilon)]/h }  \| u_h \|^2_{H_h^1( \{y; y_n \in (- 3\epsilon, - \epsilon )  \} )}  \Big). \,\,\,\,
\end{eqnarray}\\



Finally, by standard elliptic estimates, $\| u_h \|_{H_h^1} = O(1)$ and by the control assumption in Definition \ref{control}, it follows that  for any $\epsilon >0,$
$$\| u_h \|^2_{H_h^1( \{y; y_n \in (-\frac{\epsilon}{2}, 0) \} ) } \geq C_{3,N}(\epsilon) h^{N}.$$

Consequently, from (\ref{c4})-(\ref{c6}) it follows that with $h \in (0,h_0(\epsilon)],$ there exist constants $C_{j,N}(\epsilon)>0, j=3, 4, 5,$ such that

\begin{eqnarray} \label{UPSHOT}
 h^2 \| e^{\phi/h}  u_h \|_{H^1_h( supp \, \widetilde{ \partial \chi_+}) }^2  &  \geq C_2(\epsilon) e^{2m(\epsilon)/h} \Big(  h^{N+1} C_{3,N}(\epsilon) + O_{\epsilon}(e^{-2 C_3(\epsilon)/h}) \Big) \nonumber \\
& \geq C_{4,N}(\epsilon) h^{N+1}  e^{2m(\epsilon)/h}   \geq C_{5,N}(\epsilon) e^{m(\epsilon)/h}.
  \end{eqnarray}\\
  
  Next, we relate the weight function $\phi_{\epsilon}$ to Agmon distance $d_E$. From Lemma \ref{estimate} we recall that
   \begin{align} \label{relate0}
  d_E(y)  &\geq \frac{2}{3} (\min_{y \in U_E(r_0)\cap\Omega_E^c} \partial_{y_n} V )^{1/2}  \, y_n^{3/2}, \quad   y \in U_E(r_0) \cap \Omega_E^c \\ \nonumber \\ 
  &= \Big( \frac{ \min_{y \in U_E(r_0)\cap\Omega_E^c} \partial_{y_n} V  } { \max_{U_E(r_0)} \partial_{y_n} V } \Big)^{1/2} \,\phi_{\epsilon}(y_n) + O(\epsilon).  
  \end{align}\\
 
  The last estimate in (\ref{relate0}) follows since in the definition of the weight $\phi_{\epsilon}$  (see (\ref{localweight})), we choose $\tau = \max_{y \in U_E(r_0)} |\partial_{y_n} V|^{1/2}$ (see (\ref{scaling})).  Since from  the monotonicity assumption in Definition \ref{monotonicity},
  $ \min_{y \in U_E(r_0)\cap\Omega_E^c} \partial_{y_n} V \geq C(r_0) >0,$   it then follows that
 \begin{eqnarray} \label{relate}
  \phi_{\epsilon}(y_n) &\leq \Big( \frac{  \max_{y \in U_E(r_0)}  \partial_{y_n} V   }{  \min_{y \in U_E(r_0)\cap\Omega_E^c}  \partial_{y_n} V  }   \Big)^{1/2} \, d_E(y) + O(\epsilon). 
 \end{eqnarray}\\
 Thus, in view of (\ref{UPSHOT}) and (\ref{relate}), we have proved the following reverse Agmon estimate for eigenfunctions satisfying the control assumption.\\

  \begin{theo}\label{thm1}
  Let $r_0>0$ define the collar neighbourhood $U_E(r_0)$ of the hypersurface $\{ V= E\}$ as above and consider an annular subdomain   $$ A(\delta_1,\delta_2) \subset \Big( \{ V> E \} \cap U_E(r_0) \Big), \quad 0<\delta_1 < \delta_2 < r_0. $$  Then, under the control and monotonicity assumptions in Definitions \ref{control}  and \ref{monotonicity}, it follows that for any $\epsilon >0$ and  $h \in (0,h_0(\epsilon)],$ there exists a constant $ C(\epsilon,\delta_1,\delta_2) >0$ such that
  
  $$ \| e^{ \tau_0 \, d_E /h} u_h \|_{H_h^1(A(\delta_1,\delta_2))} \geq C(\epsilon,\delta_1,\delta_2) \,e^{-\beta(\epsilon)/h}, $$
  with 
  $$\tau_0 = \Big( \frac{ \max_{U_E(r_0)} \partial_{y_n} V }{  \min_{U_E(r_0)\cap\Omega_E^c} \partial_{y_n}V } \Big)^{1/2}$$
  
 \noindent and where   $ 0 < \beta(\epsilon) = O(\epsilon)$ as $\epsilon \to 0^+.$ \end{theo}\

 \begin{rem}{rem:B}{} \label{general} We note in the more general case where the eigenfunction sequence satisfies the $\epsilon_0$-control assumption, the estimate in Theorem \ref{thm1} is still valid (similarily for Theorems \ref{thm2} and \ref{thm3}). In such a case, the constant $\beta(\epsilon_0)$ can be readily estimated explicitly in terms of the potential from (\ref{relate0}) and (\ref{relate}) above. 
 
 
 \end{rem}

 


  
    \section{$L^p$ restriction lower bounds in forbidden regions} \label{restriction}

    Consider a $C^{\infty}$ hypersurface $H \subset \Omega_E^c$ in the forbidden region that bounds a domain $\Omega_H \subset \Omega_E^c$   and is {\em admissible} in the sense of Definition \ref{admissible} (ii) (see also Figure 2).   The point of this section is to extend Theorem \ref{thm1} to lower bounds for $L^2$-restrictions of eigenfunctions to hypersurfaces $H$ in the forbidden region.
    
    Let $\nu$ be the unit exterior normal to $H$ with $\langle \nabla V, \nu \rangle < 0.$ Then, under the separation assumption above, by Green's formula,
    
    \begin{equation} \label{green1}
   \int_{\Omega_H} |h \nabla u_h |_g^2 \, dv_g + \int_{\Omega_{H}} (V-E(h)) |u_h|^2 \, dv_g = h^2 \int_H \partial_{\nu} u_h \cdot u_h  \,d\sigma  \end{equation}\
   
  
  
   
  
  
  

   Using the fact that $V(x) -E \geq C>0$ for all $x \in \Omega_H,$  it follows from (\ref{green1}) that with a constant $C_{\delta} = C(V,E,E',\delta)>0$
    
    \begin{equation} \label{green4}
 h^2   \int_{H} \partial_{\nu} u_h \cdot u_h  \,d\sigma   \geq  C_{\delta}  \| u_h \|_{H^1_h(\Omega_H)}^2.\end{equation}\

 

   From the pointwise Agmon estimates   in (\ref{agmonptwise}),  for any $\delta >0,$
   $$ \| h \partial_{\nu} u_h \|_{L^\infty(H)} = O_{\delta} (e^{[-d_E(H) + \delta ]/h} ), \quad d_E(H):= \min_{q \in H} d_E(q)$$
 together with the H\"{o}lder inequality,
    
    \begin{equation} \label{green6}
      \|u_h \|_{L^p(H)}  \geq  C_{\delta}(p)  \, e^{[d_E(H) - \delta]/h} \| u_h \|^2_{H^1_h( \Omega_H) }, \quad p \geq 1.  \end{equation} \
      


     
     \begin{defn} \label{admissible}
     We say that the hypersurface $H \subset \{ V>E \}$ is {\em admissible} provided:\\
     
     (i) $H$ is a separating hypersurface bounding a $C^{\infty}$ domain $\Omega_H \subset \{ V> E \}.$\\
     
     (ii) There exists $E' > E$ such that the hypersurface $\Lambda_{E'}  = \{ y_n = E'-E \}$ has the property that 
     $$ \Lambda_{E'} \subset \Omega_H \cap U_E(r_0).$$
     \end{defn}
     
\begin{center}
\begin{tikzpicture}
\draw [fill=red!20] (-0.7,-0.3) to [out=97, in=90] (3.2, 1.7) to [out=270, in=270] (-0.7, -0.3);
\draw [thick, fill= white] (0,0) to [out=90, in=90] (2,1) to [out=270, in=270] (0,0);
\node at (1.03,0.5) {$V<E$};
\draw [very thick] (-0.4,-0.4) to [out=90, in=90] (2.5,1.5) to [out=270, in=270] (-0.4, -0.4);
\node at (2.1, 1.7) {$H$};
\node at (1.3, -1.3) {Admissible};

\draw [fill=red!20] (5.3, -0.3) to [out=97, in=90] (9.2, 1.7) to [out=270, in=270] (5.3, -0.3);
\draw [thick, fill= white] (6,0) to [out=90, in=90] (8,1) to [out=270, in=270] (6,0);
\node at (7.03, 0.5) {$V<E$};
\draw [very thick] (8,1.5) to [out=90, in=180] (8.5, 2) to [out=0, in=90] (9,1.5) to [out=270, in=270] (8, 1.5);
\node at (8.7, 1.5) {$H$};
\node at (7.2, -1.3) {Not Admissible};
\node at (4.5,-2.2) {Figure 2. \,\,\,Red region is $ \{ V> E \}\cap U_{E}(r_0)$};
\end{tikzpicture}
\end{center}
\medskip     
     
  Set   
   \begin{equation} \label{eh}
   E(H):= \inf \{ E'>E; \, \Lambda_{E'}\subset  ( \Omega_H \cap U_E(r_0) ) \}. \end{equation}\
   
   Since $\Lambda_{E'}  \cap \Omega_H = \emptyset$ for any $E' >E$ sufficiently close to $E$, it follows that $E(H) >  E. $
   Moreover, under the admissiblity assumption, it follows that for any $\delta >0$ sufficiently small  (we use the same $\delta>0$ here as in (\ref{green6})) 
   
   $$ A(E(H), E(H) + \delta) \subset ( \Omega_H \cap U_E(r_0) )$$\
   
\noindent   and so,
   
   \begin{equation} \label{green7}
    \| u_h \|^2_{H^1_h( \Omega_H) }  \geq  \| u_h \|^2_{H^1_h (A(E(H), E(H) + \delta)}. \end{equation}\
   
   From the Carleman estimate in Theorem \ref{thm1},
  
   \begin{equation} \label{green8}
   \| e^{\tau_0 d_E/h} u_h \|^2_{H^1_h (A(E(H), E(H) + \delta) )} \geq C(\delta,\epsilon) e^{- \beta(\epsilon)/h},
   \end{equation}\
   
\noindent   where $ \beta(\epsilon) \to 0^+ $ as $\epsilon \to 0^+.$  Here, we recall that  $\epsilon >0$ is the parameter appearing in the control condition in Definition \ref{control}. 
   \\
   
Since $\delta >0$ is  arbitrary, we can set $\delta = \epsilon$ in (\ref{green6})-(\ref{green8}) and then, it follows that for any $\epsilon>0,$
  and  with
    
    \begin{equation} \label{UPSHOT2}
   \tau_0 = \Big( \frac{ \max_{U_E(r_0)} \partial_{y_n} V }{  \min_{U_E(r_0)\cap\Omega_E^c}  \partial_{y_n}V } \Big)^{1/2}, \quad d_E^H := \max_{q \in \Lambda_{E(H)} } d_E(q), \quad d_E(H) = \min_{q \in H} d_E(q), \end{equation}\
   
\noindent one has the following lower bound for $L^p$-restrictions of the $u_h$ to $H:$   
    

$$  \|u_h \|_{L^p(H)}  \geq  C(\epsilon,p) \, e^{- 2 \tau_0 \cdot d_{E}^H/h }  \cdot e^{d_E(H)/h} \cdot e^{-\tilde{\beta}(\epsilon)/h}, \quad p \geq 1$$\

\noindent   where $\tilde{\beta}(\epsilon):=\beta(\epsilon) + \epsilon \to 0^+$ as $\epsilon \to 0^+.$   In summary, we have proved
   \bigskip


        


 
 
 \begin{theo} \label{thm2} Let $H$ be an admissible hypersurface in sense of Definition \ref{admissible}. Then, under the control and monotonicity assumptions in Definitions \ref{control}  and \ref{monotonicity} and  with  $E(H)$ in (\ref{eh})  and $d_E^H, d_E(H), \tau_0$ in (\ref{UPSHOT2}), it follows that for any $\epsilon >0$ and with $h \in (0,h_0(\epsilon)],$ 
 
 $$   \|u_h \|_{L^p(H)}  \geq C(\epsilon,p) \, e^{ - \, [ \, 2 \tau_0 \, d_{E}^H  - d_E(H) + \tilde{\beta}(\epsilon) \, ] / h}, \quad p \geq 1,$$\
 
 \noindent where $\tilde{\beta}(\epsilon) \to 0^+$ as $\epsilon \to 0^+.$  \end{theo}

    \begin{rem}{rem:C}{} We note that since $\tau_0 \geq 1$ and  $d_E^H \geq d_E(H)$, it is clear that the constant  $2\tau_0(H) d_E^H - d_E(H) >0.$   
    \end{rem}

    
    
  
  
  

\bigskip
    
    \section{Nodal intersection bounds in forbidden regions} \label{nodal}

     Consider the special case where $\dim M = 2$ and $(M,g,V)$ are, in addition, real-analytic. Let  $H \subset \Omega_{E}^c$ be a simple, closed, real-analytic curve in the forbidden region.
    In \cite{CT} ,the authors obtained nodal intersection bounds for the nodal sets of the eigenfunctions $u_h$ with the fixed curve, $H.$
    More precisely,   given the nodal set 
    $$ Z_{u_h} = \{x \in M;  u_h(x) = 0\},$$
the problem is to estimate the number of nodal intersections with $H$; that is 
    $ \# \{ H \cap Z_{u_h} \}$ which is just the cardinality of the intersection. Indeed, under an exponential lower bound on the $L^2$-restrictions of the eigenfunctions (ie. a {\em goodness} bound), this intersection consists of a finite set of points. 
    

 Let  $q:[0,2\pi] \to H$  be a $C^{\omega}$, $2\pi$-periodic, parametrization of $H$. 
To bound the number of zeros of $u_h \circ q: [0, 2\pi] \to \R$ we consider its holomorphic extension   $(u_h \circ q)^\C:H_\rho^\C \to \C$ to
  $$H_{\rho}^{\C} := \{ t \in \C: \; \Re t \in [0,  2\pi],  \; |\Im t| < \rho \}.$$
  
  The zeros of $(u_h \circ q)^\C$ are studied using the Poincar\'e-Lelong formula:
  $$\partial \overline{\partial} \log |(u_h \circ q)^\C(z)|^2=\sum_{z_k \in  Z_{(u_h \circ q)^\C} } \delta_{z_k}(z).$$
  
   Consider the complex tube of radius $\rho>0$ about $H$  given by  $A(\rho) := q^{\C} (H_{\rho}^{\C})$  and let $G_{\rho}$ be the Dirichlet Green's function of $A(\rho)$ satisfying:  
  
  $$ \Delta_y G_{\rho}(x,y) = \delta_x(y), \quad (x,y) \in A(\rho) \times A(\rho),$$
  $$ G_{\rho}(x,q) = 0, \quad (x,q) \in A(\rho) \times \partial A(\rho).$$
  We define the constant  
  
  \begin{equation} \label{ch} 
C_H(\rho) := \big| \, \max_{(z,w) \in A(\rho/2) \times A(\rho/2)} G_{\rho}(z,w) \, \big|^{-1}. \end{equation}
We note here that by the maximum principle $  \max_{(z,w) \in A(\rho/2) \times A(\rho/2) }  \, G_{\rho}(z,w)  < 0,$ so that $C_H(\rho) >0$ is finite. Then, by \cite[Proposition 10]{TZ},  
 \begin{equation} \label{E: [TZ]}
 \# \{ Z_{u_h} \cap H \} \leq  \# \{ Z_{(u_h \circ q)^\C} \cap H_{\rho}^{\C}\} \leq   C_H(\rho)   \, \max_{t \in H_{\rho}^{\C}} \log |F_{h}^{\C}(t)|, 
   \end{equation}

\noindent where $F_h^{\C}(t)$ with $t \in H_{\rho}^{\C}$ is the holomorphic continuation of  the normalized eigenfunction traces 
 \begin{equation}\label{E: U(t)}
  F_{h}(t):= \frac{u_h(q(t))}{\|u_h \|_{L^2(H)}}.
  \end{equation}
 It follows that we shall need to control the complexification $F_h^{\C}(t)$ to obtain upper bounds on the complex counting function $\#\{Z_{\phi_h} \cap H \} $. 
 Without loss of generality we assume that $H \subset \, \text{int} \, \Omega_\gamma $ where  $\Omega_\gamma \subset \Omega_{E}^c$ is  a domain whose closure is contained in $\Omega_E^c$  and whose boundary is a closed $C^\omega$ curve that we call $\gamma$. In particular, one has $\gamma \cap H = \emptyset$  so that with $d_E(H,\gamma):= \inf_{(q,r) \in H \times \gamma} d_E(q,\gamma),$
  $$d_E(H,\gamma) > 0, \quad d_E(H) - d_E(\gamma) >0.$$
 
 In \cite{CT} (see section 2), the authors consider the $h$-elliptic operator $ -h^2 \Delta_{g_{\Omega_E}} + 1 : C^{\infty}(M) \to C^{\infty}(M)$ where $g_{\Omega_E}$ is a metric extension to $M$ of  the Agmon metric $g_E |_{\overline{\Omega}_\gamma}$.
 
 After choosing a suitable cutoff $\chi_{\Omega_{\gamma}} \in C^{\infty}_{0}(\Omega_E^c)$ with $0 \leq \chi_{\Omega_{\gamma}} \leq 1$ and $\chi_{\Omega_{\gamma}} |_{\Omega_{\gamma}} = 1,$ one can define  the $h$-pseudodifferential operator $Q(h) =  \chi_{\Omega_{\gamma}} (-h^2 \Delta_{g_{\Omega_E}} + 1 )^{-1} \chi_{\Omega_{\gamma}}: C^{\infty}_0(\Omega_E^c) \to C^{\infty}_0(\Omega_E^c).$  The restriction of the Schwartz kernel of $Q(h)$ to $\gamma \times H$ given by $Q_{\gamma}^H(t,s,h):= Q(q(t),r(s),h)$ is then the kernel of an analytic $h$-Fourier integral operator $Q_{\gamma}^H(h): C^{\infty}(\gamma) \to C^{\infty}(H)$ where $Q_\gamma^H(t,s,h)$ is of the form
 
 \begin{eqnarray} \label{psdokernel}
 (2 \pi h)^{-2} \int_{\R^2} e^{i \langle q(t)-r(s),\xi \rangle/h} e^{-C|q(t)-r(s)|^2 \langle \xi \rangle /h} a(q(t),r(s),\xi;h) \, d\xi, \quad
  (q(t),r(s)) \in H \times \gamma,\end{eqnarray}\
  
In (\ref{psdokernel}), $a$ is an analytic semiclassical symbol of order zero and $C>0$ is an appropriate constant.  Moreover, the Schwartz kernel $Q(t,s,h)$ extends holomorphically in the outgoing $t$ variable to $Q^{\C}(t,s)$ with $(t,s) \in H_{\rho^*}^{\C} \times \gamma$ where $\rho^* >0$ is a suitable tube radius independent of $h$ (see Remark \ref{rem:D} (ii) and (iii) below).   It then follows from a potential layer formula \cite{CT} (14), that the eigenfunction restriction $u_h \circ q(t)$ also holomorphically continues to $(u_h \circ q)^{\C}(t)$ in the strip $H_{\rho^*}^{\C}.$ Setting
 $$ {\bf Q}(h,\rho^*):=  \max_{(r(s),q^{\C}(t)) \in \gamma \times H_{\rho^*}^{\C}} ( |Q^{\C}(t,s)|, | \partial_{\nu(s)} Q^{\C}(t,s)) | )$$
and using the formula in (\ref{psdokernel}) combined with some potential layer analysis,
in \cite{CT} (16), it is proved that
 
 \begin{eqnarray} \label{forbidden}
 |F_h^{\C}(t)| \leq C \, {\bf Q}(h,\rho^*) \left( \frac{ \| u_h \|_{L^2(\gamma)} }{ \|u_h \|_{L^2(H)} } + \frac{ \| \partial_{\nu} u_h \|_{L^2(\gamma)} }{ \| u_h \|_{L^2(H)} } \right), \quad t \in H_{\rho^*}^{\C}.
 \end{eqnarray}

In \cite{CT} Theorem 4, the authors show that  by choosing $\rho^* >0$ sufficiently small (see Remark \ref{rem:D} (iii)) one can show that

\begin{equation} \label{qbound}
 {\bf Q}(h,\rho^*) = O(1). \end{equation}

Thus, from (\ref{forbidden}) this gives (see Remark \ref{rem:D}) ),

  \begin{eqnarray} \label{Forbidden}
 |F_h^{\C}(t)|  \leq  C  \,  \left( \frac{ \| u_h \|_{L^2(\gamma)} }{ \|u_h \|_{L^2(H)} } + \frac{ \| \partial_{\nu} u_h \|_{L^2(\gamma)} }{ \| u_h \|_{L^2(H)} } \right), \quad t \in H_{\rho^*}^{\C} 
 \end{eqnarray}

 It follows from the Agmon estimates in (\ref{agmon}) that for any $\epsilon >0,$ 
 
 $$\max \{  \, \| u_h \|_{L^2(\gamma)}, \, \| \partial_{\nu} u_h \|_{L^2(\gamma)} \, \} \leq C(\epsilon) e^{[- \, d_E(\gamma) + \epsilon] \, /h }, \quad d_E(\gamma) = \min_{r \in \gamma} d_E(r)$$ 
for all $h \in (0, h_0(\epsilon)].$  On the other hand, if $H$ is admissible in sense of Definition \ref{admissible}, from Theorem \ref{thm2}, we have the lower bound
$$ \|u_h \|_{L^2(H)} \geq  C(\epsilon) e^{ [ \, - 2 \tau_0  d_E^H  + d_E(H) - \tilde{\beta}(\epsilon)  \, ]/h}, \quad d_E^H = \max_{q \in \Lambda_{E(H)} } d_E(q), \,\, \tilde{\beta}(\epsilon) = o(1).$$

Consequently, from (\ref{Forbidden}) we get that with $\tilde{\beta}(\epsilon) = o(1)$ as $\epsilon \to 0^+,$

\begin{eqnarray} \label{forbidden2}
 |F_h^{\C}(t)| \leq C(\epsilon)   e^{ [ \tilde{\beta}(\epsilon) + \epsilon] /h } \cdot e^{ [ \,2 \tau_0 d_E^H - d_E(H) - d_E(\gamma)  \, ]/h},
  \quad t \in H_{\rho^*}^{\C}.  \end{eqnarray}\\
 
 
 Finally, taking $\log$ of both sides of (\ref{forbidden2}) gives for  $h \in (0,h_0(\epsilon)],$

 \begin{equation} \label{logbound}
  h \,  \log |F_h^{\C}(t)| \leq  2 \tau_0 d_E^H - d_E(H)  - d_E(\gamma)+ \mu(\epsilon),
  \end{equation}\
  
\noindent with $\mu(\epsilon):= \tilde{\beta}(\epsilon) + \epsilon  \to 0^+$ as $\epsilon \to 0^+.$  A combination of (\ref{E: [TZ]}) and (\ref{logbound}) then proves the following
 

\begin{theo} \label{thm3}  Assume that $\dim M = 2$ and $(M,g,H,\gamma,V)$ are all real-analytic  and suppose that $H$ is an admissible hypersurface in sense of Definition \ref{admissible}. Then, under the control and monotonicity assumptions in Definitions \ref{control}  and \ref{monotonicity}, for any fixed $\epsilon >0,$ there are constants $\mu(\epsilon) >0$ and $h_0(\epsilon)>0$ such that for $h \in (0,h_0(\epsilon)],$


$$  \# \{ Z_{u_h} \cap H \} \leq C_H(\rho^*) \cdot \Big(  2  \tau_0 d_E^H - d_E(H)  - d_E(\gamma) + \mu(\epsilon)   \Big) \, h^{-1}.$$ 

\noindent Here, $ \mu(\epsilon)  \to 0^+$ as $\epsilon \to 0^+$ and $\tau_0, d_E^H, d_E(H)$ are the constants defined in (\ref{UPSHOT2}). In addition, $\rho^* >0$ is the tube radius  and $C_H(\rho^*)$ is the corresponding Dirichlet Green's function constant defined in (\ref{ch}).
\end{theo}

\begin{rem}{rem:D}{} (i)  We note that in Theorem \ref{thm3} above, the constant on the RHS given by    $C_H(\rho^*) \cdot \big(  2  \tau_0 d_E^H - d_E(H) - d_E(\gamma) + \mu(\epsilon)   \big) = C_H(\rho^*) \cdot \big(  2 ( \tau_0 d_E^H - d_E(H) )  + d_E(H) - d_E(\gamma) + \mu(\epsilon)   \big) >0.$ This follows since $ \tau_0 d_E^H -  d_E(H) \geq d_E^H -  d_E(H) \geq 0,  \, d_E(H) - d_E(\gamma) >0$ and the Green's function constant $C_H(\rho^*)>0.$  \\

\noindent (ii)  As we have pointed out above,  tube radius $\rho^*>0$  is determined directly from analytic continuation properties of the operator kernel $Q_{\gamma}^H(t,s,h)$ and, as such, depends on the analyticity properties of the data $(M,g,V,H,\gamma)$ in a complicated way (see also (iii) below). However, $\rho^*$ and consequently,  the Green's function constant $C_H(\rho^*),$ is nevertheless universal; that is, it is determined independent of the eigenfunction sequence $u_h.$  On the other hand, the part  of the nodal bound in (\ref{E: [TZ]}) that depends on the $u_h$'s   (i.e. $\max_{t \in H_{\rho*}^{\C}} \log |F_h^{\C}(t)|$ ) is bounded in (\ref{logbound})  rather explicitly in terms of the Agmon distance function.  \\

\noindent (iii)  It follows from \cite{CT} Theorem 4 that the bound in (\ref{qbound}) can be improved  to 
\begin{equation} \label{qbound2}
{\bf Q}(h,\rho^*) = O(e^{-C_1(\rho^*)/h})
\end{equation}
for some $C_1(\rho^*)>0.$ It follows from the holomorphic continuation properties  of the phase in (\ref{psdokernel}) that when $d_E(\gamma,H) = \inf_{(r,q) \in \gamma \times H} d_E(r,q)$ is sufficiently small, one can choose 

$$ \rho^* = \frac{1}{C_2} d_E^{\,2}(\gamma,H)$$

\noindent where $C_2>0$ is sufficiently large.  For such a choice of $\rho^*$, the constant in (\ref{qbound2}) can be taken to be 

$$ C_1(\rho^*) = \frac{1}{C_3} d_E^{\,2}(\gamma,H),$$

\noindent where $C_3>0$ is another constant.

However,  since 
$$ d_E(H) - d_E(\gamma) - \frac{1}{C_3} d_E^{\,2}(\gamma,H) \sim d_E(H) - d_E(\gamma)\quad  \text{as}  \,\,\, d_E(\gamma,H) \to 0^+,$$

\noindent (\ref{qbound2}) only gives a marginal improvement  in Theorem \ref{thm3}. Moreover, both constants $C_2$ and $C_3$ above depend in a complicated way on the data $(M,g,V,H)$ and are difficult to determine explicitly. Consequently, we decided to omit these considerations from the statement of Theorem \ref{thm3}.

\end{rem}




\bigskip

  \section{Eigenfunction control condition: examples} \label{counterexample}



\subsection{Counterexample:  lack of eigenfunction control} 
Here we show that without the control assumption in Definition \ref{control}, we can establish a  Schr\"odinger  model such that the corresponding eigenfunction decays much faster than $e^{-(1-\epsilon) d_E/h}$ in $A(\delta,\delta')$ for $\delta'$ small enough. Such counterexample is essentially inspired by the paper (\cite{CT}).

 Consider a convex surface of revolution generated by rotating a curve $ \gamma = \{ (z, f(z)), \, z \in [-1,1] \}$ about the $z$-axis where $  f\in C^\infty ( (-1,1), \R^+) \cap C^{0}( [-1,1]; \R^+),$ $f(1)=f(-1)=0$ and both $f(z) >0$ and $f''(z)<0$ for all $z\in (-1,1).$  
 Let $M$ be the corresponding convex surface of revolution given by the parametrization
 
 \begin{align*}
 &\beta: [-1,1]\times [0, 2\pi) \to \R^3, \\
 &\beta(z,\theta) = (f(z) \cos \theta, f(z) \sin \theta, z).
 \end{align*}

 In addition, we require here that 
    
 $$ f^2 \in C^{\infty}([-1,1];\R) \quad \text{and} \,\, \lim_{z \to \pm 1^{\mp}}  (f^2)'(z)\newline = 2 \lim_{z \to \pm 1^{\mp}}  f (z) f'(z) \neq 0.$$
  
The latter condition ensures that $M$ is smooth near the poles.  Indeed, since $f^2 \in C^{\infty}([-1,1]),$ by the Whitney extension theorem \cite{Wh}, it has a $ C^{\infty}$-extension in $(-1-\delta, 1 + \delta)$ for  $\delta >0.$ We abuse notation slightly and denote the extended function also by $f^2.$  Then, since $ x^2 + y^2 =  f^2(z),$ by an application of the implicit function theorem, it follows that if $(f^2)' (\pm 1) = 2 f f' (\pm 1) \neq 0,$ there exist  $F_{\pm} \in C^{\infty}_{loc}$ such that $z = F_{\pm}(x^2 + y^2)$ for $(x,y,z) \in M$ near the respective poles $(0,0,\pm1).$ In particular, $M$ is then locally a smooth graph over the $(x,y)$-plane near the poles.  Also, since $f^2(r) > 0$ for $r \in (-1,1)$ and $f^2(\pm 1) =0,$ it follows that $\mp (f^2)'(\pm 1) >0.$ 
 
 As an example,  in the case of the round sphere, $f(z) = \sqrt{1-z^2}, z \in [-1,1],$ so that $f^2(z) = 1- z^2$ and $ \lim_{z \to \pm 1^{\mp}}  f (z) f'(z) = \mp 1. $  Written in terms of standard spherical coordinates, $z = \cos \phi$ and $f(z) = \sin \phi,$ where $\phi \in [0,\pi]$ is the polar angle with the $z$-axis.

Then, $M$ inherits a Riemannian metric $g$ given by
 \[g=w^2(z)dz^2+f^2(z)d\theta^2,\]
where $w(z)=\sqrt{1+(f'(z))^2}$. 

Let $E \in \R$ be a regular value of $V$ and consider the Schr\"odinger equation on $M$ given by
\[(-h^2 \Delta_g + V)\phi_h=E(h)\phi_h,\]
where $V \in C^{\infty}(M)$ and is  axisymmetric,  so that $V(z,\theta)=V(z).$ We also assume that $E(h) = E + o(1)$  and that $ \{ V = E \}$ is disjoint from the poles $(0,0,\pm 1).$ 

 In the following, we will relabel the $z$-coordinate and write $r = z$ to indicate that this variable essentially measures distance into the  forbidden region $\{ V(r) > E \}.$

We seek eigenfunctions of the form $\phi_h(r,\theta)=v_h(r)\psi_h(\theta).$ These are joint eigenfunctions of the quantum integrable system given by the commuting operators $H = -h^2 \Delta_g + V(r)$ and $P = h D_{\theta}$ \cite{TZ2}. The Laplace operator in the coordinates $(r,\theta)$ has the following form 
\[ \Delta_g= \frac{1}{w(r)f(r)}\frac{\partial}{\partial r}\left( \frac{f(r)}{w(r)}\frac{\partial}{\partial r} \right)+  \frac{1}{f^2(r)} \frac{\partial^2}{\partial\theta^2}.\] 

It follows that  $\psi_h(\theta)$ must satisfy the ODE
\begin{equation} \label{tang}
- h_k^2 \frac{d^2}{d\theta^2}\psi_h(\theta)= h_k^2 \,m_{h_k}^2 \psi_h(\theta)
\end{equation}

Let $\{h_k\}$ be a decreasing sequence with $h_k\to 0^+$ as $k\to +\infty$ and $m_{h_k}=1/h_k \in \Z$. Then, we choose a particular sequence of solutions of \eqref{tang} given by
\[\psi_{h_k}(\theta)=e^{i m_{h_k} \theta}.\]

Then, $v_{h_k}(r)$  must satisfy

\begin{equation} \label{veqn}
-h_k^2 \, \frac{1}{f^{2}(r)}  \left( \frac{f(r)}{w(r)} \frac{d}{dr} \right)^2 v_{h_k}(r) + [ \,  ( V(r) - E ) +  \frac{1}{f^2(r)}  \, ]  \, v_{h_k}(r) = 0, \quad r \in (-1,1).\end{equation}

As for boundary conditions in (\ref{veqn}) (recalling that $r = z $), the eigefunctions are of the form
\begin{equation} \label{basicform}
 e^{im_{h_k} \theta} v_{h_k}(z) = (\cos \theta + i \sin \theta)^{m_{h_k}} \, v_{h_k}(z) = \frac{ (x+iy)^{m_{h_k}} }{ f(z)^{m_{h_k}}}  \,  v_{h_k}(z). \end{equation}
 
 
 Thus, in view of (\ref{basicform}), to ensure that the eigenfunctions $\phi_{h_k}(z,\theta) = v_{h_k}(z) \psi_{h_k}(\theta)$ are smooth up to the poles $z = \pm 1,$  the $v_{h_k}$  in (\ref{veqn}) must be of the form
 
 $$v_{h_k}(z) = f(z)^{m_{h_k}}  \cdot v_{h_k}^{reg} (z), \quad v_{h_k}^{reg} \in C^{\infty}([-1,1]).$$
 
 Thus, since $m_{h_k} \neq 0$ here ($m_{h_k} \to \infty$) and $f(\pm 1) = 0,$ it follows that the $v_{h_k}$ in (\ref{veqn})  satisfy the Dirichlet boundary conditions $v_{h_k}(\pm 1) = 0.$

 We note that this is precisely what happens in the case of standard spherical harmonics, where $V = 0, E =1, f(z)  = \sqrt{1-z^2}$ and the $v_{h_k}$ are the associated Legendre polynomials.

Next, we make the radial change of variables $s\to r(s)=\int_0^s\frac{f(\tau)}{w(\tau)} d\tau.$ In view of the assumption on profile function above (ie. $f^2 \in C^{\infty}[-1,1]$ and $ \mp (f^2)'(\pm 1) > 0),$ it follows that near $r=1,$  we have $f(r) \sim c_1 (1- r)^{1/2}$ and $f'(r) \sim - c_2 (1-r)^{-1/2}$ with $c_j >0, j=1,2.$ Thus,  
$ \frac{dr}{ds}  = \frac{f(r)}{\sqrt{ 1 + |f'(r)|^2} } \sim c_3 (1-r)$  and so,
$ s(r) \sim - \log (1-r)$ as $ r \to 1^-.$ Consequently, for the inverse function $r \mapsto s(r),$ we have $s(1) = \infty.$ Similar reasoning at $r = -1$ shows that $s(-1) = -\infty.$

Setting $\tilde{v}_{h_k}(s):= v_{h_k}(r(s)),$ the equation (\ref{veqn}) becomes


\begin{equation} \label{reduced}
\left(-h_k^2\frac{d^2}{ds^2}+f^2(r(s))(V(r(s))-E(h))+ 1 \right) \tilde{v}_{h_k}(s) =0, \quad s \in \R,
\end{equation} 
where  $\tilde{v}_{h_k}(\pm \infty) = 0.$ 

Consider the effective potential 

 $$V_{eff}(r):= f^2(r) ( V(r) - E ) + 1,$$\ 

\noindent where $r \in [-1,1].$ 


We assume here that $V(r=0) = E.$ Then, since $V_{eff}(0) = 1,$ it follows that there exists $\alpha >0$ and a corresponding annulus $A(-\alpha, \alpha)= \{r; -\alpha< r<  \alpha\}$ centered around $\{r=0 \} \subset \{ V=E\}$ disjoint from the poles,  so that

\begin{equation} \label{eff1}
V_{eff}(r) \geq \frac{1}{4}, \quad \forall r \in A(-\alpha,\alpha). \end{equation}

Consider the annulus $A(-\epsilon_0,  \epsilon_0) = \{r; -\epsilon_0< r<  \epsilon_0\}$ where $\epsilon_0 \ll \alpha.$ Then, in view of (\ref{eff1}), it then follows by the  standard Agmon-Lithner estimate applied to (\ref{reduced}) that for any $\delta >0,$



$$ \| e^{(1-\delta) \int_{ r^{-1}(-\alpha)}^{s} \sqrt{V_{eff}(r(s))}  \, ds / h_k } \, \tilde{v}_{h_k}(s) \|_{ L^2( A(r^{-1}(-\epsilon_0),r^{-1}(\epsilon_0) ) )} = O_{\delta}(1). $$
Making the change of variables $s \to  r(s)$ then gives

\begin{eqnarray}\label{counter0}
\| e^{(1-\delta) \int_{-\alpha}^{r} \frac{ \sqrt{V_{eff}(\tau)} }{\partial_s \tau} \, d\tau / h_k } \, v_{h_k}(r) \|_{L^2( A(-\epsilon_0,\epsilon_0) )} = O_{\delta}(1), \end{eqnarray}

In view of (\ref{eff1}) and the fact that $0<\partial_s r = \frac{f(r(s))}{w(r(s))}  \leq C_1$  for all $r(s) \in A(-\alpha,\alpha)$ follows from (\ref{counter0}) that

\begin{equation}\label{counter}
\|e^{ \frac{(1-\delta)}{2 C_1} \, (\alpha + r) / h_k} v_{h_k}(r)\|_{ L^2 (A(-\epsilon_0,\epsilon_0) )} = O_\delta(1).
\end{equation}\


But then,   for $r \in A(-\epsilon_0,\epsilon_0)$ with $\epsilon_0 >0$ sufficiently small  (i.e. $\epsilon_0 \ll \alpha$),
the inequality \eqref{counter} contradicts the control condition in Definition \ref{control}; indeed,  the eigenfunctions already decay exponentially in $h$ in the allowed region $A(-\epsilon_0, 0).$ 

  We note that  since $d_E(r) = O( |V(r) - E|^{3/2}) = O( \epsilon_0^{3/2})$ for $r \in A(0,\epsilon_0),$   it follows that for  $\epsilon _0 >0$ sufficiently small,  in the forbidden region where $r \in A(0, \epsilon_0),$
$$ \frac{1}{2C_1}(\alpha + r)     > \tau_0 d_E(r).$$

In this case, the exponential decay is therefore more pronounced than in Theorem \ref{thm1}. This is due to the presence of the effective potential term $m_k h_k \sim 1$ which in turn appears because of the particular choice of the sequence of Fourier modes in (\ref{tang}). This is consistent with our results, since as we have already shown, the control condition is violated for this particular sequence of eigenfunctions.
\subsection{Examples of eigenfunction sequences satisfying control}



We consider precisely the same example of a Schr\"{o}dinger operator on a convex surface of rotation as above but choose the quantum number $m = const. \neq 0$ so that $m h_k = O(h_k)$ as $h_k \to 0.$ Then, the ODE in (\ref{reduced}) becomes

\begin{equation} \label{reduced2}
\left(-h_k^2\frac{d^2}{ds^2}+f^2(r(s))(V(r(s))-E(h))+ O(h_k^2) \right)v_{h_k}(r(s))=0.
\end{equation}

The fact that the corresponding eigenfunctions $\phi_h(r,\theta) = v_h(r) \psi_h(\theta)$ satisfy the control assumption is then an immediate consequence of standard WKB theory applied to the semiclassical ODE (\ref{reduced2}). Indeed, writing $\Phi(r):= \int_{r_0}^{r}  \frac{f(r)}{\partial_s r} \,( E - V(r) )^{1/2} \, dr,$ it follows by WKB asymptotics that for $r \in [-1,1]$ satisfying $ E-2\epsilon < V(r) < E-\epsilon,$
\begin{equation} \label{wkb}
v_h(r) \sim_{h \to 0^+} e^{i\Phi(r)/h} c_1(h) \, a_1(r;h) + e^{-i\Phi(r)/h} c_2(h) a_{2}(r;h), \end{equation}
where for $k=1,2,$  $a_{k}(r;h) \sim \sum_{j=0}^{\infty} a_{k,j}(r) h^j$ and
$$|c_1(h)|^2 + |c_2(h)|^2 \geq C_1 >0, \quad |a_k(r;h)| \geq C_2(\epsilon) >0; \,\, k =1,2.$$
Consequently, from (\ref{wkb}) we get that for any $\epsilon >0,$
$$\int_{ - 2\epsilon < V(r) -E < -\epsilon} \int_{0}^{2\pi} |\phi_h(r,\theta)|^2 \, dr d\theta = \int_{ - 2\epsilon < V(r) -E < -\epsilon} \int_{0}^{2\pi} |v_h(r)|^2 \, |e^{im\theta}|^2 \, dr d\theta $$
$$ = \int_{ - 2\epsilon < V(r) -E < -\epsilon} \int_{0}^{2\pi} |v_h(r)|^2 \, dr  d\theta \geq C(\epsilon)>0.$$
In the last estimate, to control mixed terms, we have used that by an integration by parts,
$$ \int_{ - 2\epsilon < V(r) -E < -\epsilon} e^{\pm 2 i \Phi(r)/h} a_1(r;h) a_2(r;h) \, dr = O_{\epsilon}(h).$$
As a result, this particular sequence clearly satisfies the control assumption in Definition \ref{control} with $N=0.$


\end{document}